\documentclass[10pt]{article}

\usepackage{graphicx}
\usepackage{amsmath,amsthm,amsfonts,amssymb,amscd,enumerate,pictex}%,refcheck}
\theoremstyle{plain}
\newtheorem{SConjecture}[subsection]{Singer's Conjecture}
\newtheorem{subSCConjecture}[subsection]{Singer's Conjecture for Coxeter groups}

\newtheorem{Mainthm}[subsection]{Main Theorem}
\newtheorem{Theorem}[subsection]{Theorem}

\newtheorem{Lemma}[subsection]{Lemma}

\theoremstyle{definition}

\DeclareMathOperator{\Vol}{Vol}

\newcommand{\cs}{\mathcal{S}}

\newcommand{\cH}{\mathcal{H}}

\newcommand{\Ltwo}{L^2}
\newcommand{\mfh}{\mathfrak{h}}

\def\l{\operatorname{\ell}}
\newcommand{\ltwo}{\l^2}

\newcommand{\St}{\operatorname{St}}

\newcommand{\Edge}{\operatorname{Edge}}

\newcommand{\gS}{\Sigma}

\newcommand{\ga}{\alpha}
\newcommand{\gb}{\beta}

\newcommand{\BS}{\mathbb{S}}
\newcommand{\BH}{\mathbb{H}}

\newcommand{\BR}{\mathbb{R}}
\newcommand{\BE}{\mathbb{E}}
\newcommand{\BN}{\mathbb{N}}

\newenvironment{enumeratei}{\begin{enumerate}[\upshape (i)]}
        {\end{enumerate}}
\newenvironment{enumeratea}{\begin{enumerate}[\upshape 
(a)]}{\end{enumerate}}

\numberwithin{equation}{section}
\begin{document}

%Title
\title{On the three-dimensional Singer Conjecture for Coxeter groups}

\author{Timothy A. Schroeder}

\date{June 30, 2009}
\maketitle
%*******

\begin{abstract}
We give a proof of the Singer conjecture (on the vanishing of reduced $\ltwo$-homology except in the middle dimension) for the Davis Complex $\gS$ associated to a Coxeter system $(W,S)$ whose nerve $L$ is a triangulation of $\BS^2$.  We show that it follows from a theorem of Andreev, which gives the necessary and sufficient conditions for a classical reflection group to act on $\BH^3$. 
\end{abstract}

\section{Introduction}\label{s:intro}
%Let $S$ be a finite set of generators.  A \emph{Coxeter matrix} on $S$ is a symmetric $S\times S$ matrix $M=(m_{st})$ with entries in $\BN\cup\{\infty\}$ such that each diagonal entry is $1$ and each off diagonal entry is $\geq 2$.  The matrix $M$ gives a presentation of an associated \emph{Coxeter group} $W$:
%\begin{equation}\label{e:coxetergroup}
%	W=\left\langle S\mid (st)^{m_{st}}=1, \text{ for each pair } (s,t) \text{ with } m_{st}\neq\infty\right\rangle.
%\end{equation}
Let $(W,S)$ denote a \emph{Coxeter system}: $S$ is a finite set of generators and for any pair $\{s,t\}$ of generators there is a particular relation $m_{st}\in\BN\cup\{\infty\}$ such that $(st)^{m_{st}}=1$ with the rule that $m_{st}=1$ if and only if $s=t$; these are the only relations.  (See \cite{humphreys} or \cite{davisbook}).  Denote by $L$ the nerve of $(W,S)$.  ($L$ is a simplicial complex with vertex set $S$, the precise definition will be given in section \ref{ss:davis}.)  In several papers (e.g., \cite{davisannals}, \cite{davisbook}, and \cite{davismoussong}), M. Davis describes a construction which associates to any Coxeter system $(W,S)$, a simplicial complex $\gS(W,S)$, or simply $\gS$ when the Coxeter system is clear, on which $W$ acts properly and cocompactly.  The two salient features of $\gS$ are that (1) it is contractible and (2) that it admits a cellulation under which the nerve of each vertex is $L$.  It follows that if $L$ is a triangulation of $\BS^{n-1}$, $\gS$ is an $n$-manifold.

The following conjecture is attributed to Singer.  
\begin{SConjecture}\label{conj:singer} If $M^{n}$ is a closed aspherical manifold, then the reduced $\ltwo$-homology of $\widetilde{M}^n$, $\cH_{i}(\widetilde{M}^{n})$, vanishes for all $i\neq\frac{n}{2}$.
\end{SConjecture}
\noindent For details on $\ltwo$-homology theory, see \cite{davismoussong}, \cite{do2} and \cite{eckmann}.

Now, if $G$ is a torsion-free subgroup of finite index in $W$, then $G$ acts freely on $\gS$ and $\gS/G$ is a finite complex.  By $(1)$, $\gS/G$ is aspherical.  Hence, if $L$ is homeomorphic to an $(n-1)$-sphere, Davis' construction gives examples of closed aspherical $n$-manifolds and Conjecture \ref{conj:singer} for such manifolds becomes the following.

\begin{subSCConjecture}\label{conj:singerc} Let $(W,S)$ be a Coxeter group such that its nerve, $L$, is a triangulation of $\BS^{n-1}$.  Then $\cH_{i}(\gS)=0$ for all $i\neq\frac{n}{2}$.
\end{subSCConjecture}

Conjecture \ref{conj:singer} holds for elementary reasons in dimensions $\leq 2$.  In \cite{do2}, Davis and Okun show that \ref{conj:singerc} holds for $n=3$ when $(W,S)$ is \emph{right-angled} (this means that generators either commute, or have no relation).  They do this in (at least) two ways, one of which is a direct calculation of the reduced $\ltwo$-homology using a Mayer-Vietoris argument (Chapter 10).  We follow that method here, proving the result for \emph{arbitrary} Coxeter systems with nerve $\BS^2$.  This paper is a precursor to a JSJ-decomposition for three-dimensional Davis manifolds, which the author details in \cite{schroedergeom}, and from which Conjecture \ref{conj:singerc} follows as a Corollary.  Also, in \cite{schroedereven}, he uses the three-dimensional case to establish \ref{conj:singerc} in the case $(W,S)$ is \emph{even} and $L$ is a flag triangulation of $\BS^3$.      

%In \cite[Theorem 2]{andreev2}, E. Andreev gives the necessary and sufficient conditions for an abstract $3$-dimensional polytope $P$ with assigned dihedral angles to be realized as a convex polytope in $\BH^3$.  Andreev's theorem is a special case (an orbifoldal version of) of Thurston's conjecture.  The purpose of this paper is to show explicitly that Conjecture \ref{conj:singerc} in dimension $3$ follows from Andreev's theorem.  We do this using methods similar to those in \cite[Section 10]{do2}, which proves the same for the right-angled case.
%
\section{The Davis complex and $\ell^2$-homology}
Let $(W,S)$ be a Coxeter system.  Given a subset $U$ of $S$, define $W_{U}$ to be the subgroup of $W$ generated by the elements of $U$.  A subset $T$ of $S$ is \textit{spherical} if $W_T$ is a finite subgroup of $W$.  In this case, we will also say that the subgroup $W_{T}$ is spherical.  Denote by $\cs$ the poset of spherical subsets of $S$, partially ordered by inclusion.  Given a subset $V$ of $S$, let $\cs_{\geq V}:=\{T\in \cs|V\subseteq T\}$.  Similar definitions exist for $<, >, \leq$.  For any $w\in W$ and $T\in \cs$, we call the coset $wW_{T}$ a \emph{spherical coset}.  The poset of all spherical cosets we will denote by $W\cs$.  

\subsection{The Davis complex}\label{ss:davis}
Let $K=|\cs|$, the geometric realization of the poset $\cs$.  It is a finite simplicial complex.  Denote by $\gS(W,S)$, or simply $\gS$ when the system is clear, the geometric realization of the poset $W\cs$.  This is the Davis complex.  The natural action of $W$ on $W\cs$ induces a simplicial action of $W$ on $\gS$ which is proper and cocompact.  $\gS$ is a model for $\underline{EW}$, a \emph{universal space for proper $W$-actions}.  (See Definition \cite[2.3.1]{davisbook}.)  $K$ includes naturally into $\gS$ via the map induced by $T\rightarrow W_{T}$.  So we view $K$ as a subcomplex of $\gS$, and note that $K$ is a strict fundamental domain for the action of $W$ on $\gS$.  

The poset $\cs_{>\emptyset}$ is an abstract simplicial complex.  This simply means that if $T\in\cs_{>\emptyset}$ and $T'$ is a nonempty subset of $T$, then $T'\in \cs_{>\emptyset}$.  Denote this simplicial complex by $L$, and call it the \emph{nerve} of $(W,S)$.  The vertex set of $L$ is $S$ and a non-empty subset of vertices $T$ spans a simplex of $L$ if and only if $T$ is spherical.  Define a labeling on the edges of $L$ by the map $m:\Edge(L)\rightarrow \{2,3,\ldots\}$, where $\{s,t\}\mapsto m_{st}$.  This labeling accomplishes two things: (1) the Coxeter system $(W,S)$ can be recovered (up to isomorphism) from $L$ and (2) the $1$-skeleton of $L$ inherits a natural piecewise spherical structure in which the edge $\{s,t\}$ has length $\pi-\pi/m_{st}$.  $L$ is then a \emph{metric flag} simplicial complex (see Definition \cite[I.7.1]{davisbook}).  This means that any finite set of vertices, which are pairwise connected by edges, spans a simplex of $L$ if an only if it is possible to find some spherical simplex with the given edge lengths.  In other words, $L$ is ``metrically determined by its $1$-skeleton.''  

For the purpose of this paper, we will say that labeled (with integers $\geq 2$) simplicial complexes are \emph{metric flag} if they correspond to the labeled nerve of some Coxeter system.  We will often indicate these complexes simply with their $1$-skeleton, understanding the underlying Coxeter system and Davis complex.  We write $\gS_L$ to denote the Davis complex associated to the nerve $L$ of $(W,S)$.  

\noindent
\textbf{Special subcomplexes.}  Suppose $A$ is a full subcomplex of $L$.  Then $A$ is the nerve for the subgroup generated by the vertex set of $A$.  We will denote this subgroup by $W_{A}$.  (This notation is natural since the vertex set of $A$ corresponds to a subset of the generating set $S$.)  Let $\cs_{A}$ denote the poset of the spherical subsets of $W_A$ and let $\gS_{A}$ denote the Davis complex associated to $(W_{A},A^{0})$.  The inclusion $W_{A}\hookrightarrow W_{L}$ induces an inclusion of posets $W_{A}\cs_{A}\hookrightarrow W_{L}\cs_{L}$ and thus an inclusion of $\gS_{A}$ as a subcomplex of $\gS_{L}$.  Such a subcomplex will be called a \emph{special subcomplex} of $\gS_{L}$.  Note that $W_{A}$ acts on $\gS_{A}$ and that if $w\in W_{L}-W_{A}$, then $\gS_{A}$ and $w\gS_{A}$ are disjoint copies of $\gS_{A}$.  Denote by $W_{L}\gS_{A}$ the union of all translates of $\gS_{A}$ in $\gS_{L}$.  

\noindent
\textbf{A mirror structure on $K$.} If $L$ is the triangulation of an $n$-sphere, then we have a another cellulation of $K$ and $\gS$.  For each $T\in\cs$, let $K_T$ denote the geometric realization of the subposet $\cs_{\geq T}$.  $K_T$ is a triangulation of a $k$-cell, where $k=n+1-|T|$.  We then define a new cell structure on $K$ by declaring the family $\{K_T\}_{T\in\cs}$ to be the set of cells in $K$.  We write $K_L$ to indicate $K$ equipped with this cellulation and note that it extends to a cellulation of $\gS_L$.  Since our concern is the case $L$ is a triangulation of $\BS^2$, we assume this cellulation of $\gS_L$.  

The boundary complex of $K_L$ is combinatorially dual to $L$, so $K_L$ has codimension $1$ faces corresponding the elements of $S$.  In fact, if $L$ is \emph{any} cell complex homeomorphic to $\BS^2$, in the strict sense that any non-empty intersection of two cells is a cell, then $L$ is combinatorially dual to the boundary complex of a $3$-dimensional convex polytope, which we will denote by $K_{L}$.  If the edges of $L$ are labeled with integers $\geq 2$, (e.g. $L$ is the labeled nerve of a Coxeter system) then we assign dihedral angles to $K_L$ so that the angle between faces dual to vertices $s$ and $t$ is $\pi/m_{st}$, where $m_{st}$ is the label on the edge between $s$ and $t$.  This assignment defines a classical reflection group generated by the reflections in the faces of $K_{L}$ with relations prescribed by the dihedral angles.  

\noindent
\textbf{A cellulation of $\gS$ by Coxeter cells.}  $\gS$ has a coarser cell structure: its cellulation by ``Coxeter cells.''  (References for this section include \cite{davisbook} and \cite{do2}.)  
%Suppose that $T\in \cs$; then by definition $W_{T}$ is finite.  Take the canonical representation of $W_{T}$ on $\BR^{\Card(T)}$.  Choose a point $x$ in the interior of a fundamental chamber.  The \emph{Coxeter cell of type $T$} is defined as the convex hull $C$, in $\BR^{\Card(T)}$, of $W_{T}x$ (a generic $W_{T}$-orbit).  The vertices of $C$ are in 1-1 correspondence with the elements of $W_{T}$.  Furthermore, a subset of these vertices is the vertex set of a face of $C$ if and only if it corresponds to the set of elements in a coset of the form $wW_{V}$, where $w\in W_{T}$ and $V\subset T$.  Hence, the poset of non-empty faces of $C$ is naturally identified with the poset
%\begin{equation*}%\label{e:finite}
%	W_{T}\cs_{\leq T}:= \{wW_{V}|w\in W_{T}, V\subset T\}.
%\end{equation*}
%Therefore, we can identify the simplicial complex $\gS(W_{T},T)$ with the barycentric subdivision of the Coxeter cell of type $T$. 
%
%Now, for each $T\in \cs^{(k)}$ and $w\in W$, the poset $W\cs_{\leq wW_{T}}$ is isomorphic to the poset $W_T\cs_{\leq T}$ via the map $vW_{V}\rightarrow w^{-1}vW_{V}$.  Thus, the subcomplex of $\gS(W,S)$ which is obtained from the poset $W\cs_{\leq wW_{T}}$ may be identified with the barycentric subdivision of the $k$-cell of type $T$.  In this way, we put a cell structure on $\gS$ which is coarser than the simplicial structure by identifying each simplicial subcomplex $|W\cs_{\leq wW_{T}}|$ with a cell of type $T$.
%We will write $\gS_{cc}$, when necessary, to denote the Davis complex equipped with this cellulation.  
The features of the Coxeter cellulation are summarized by \cite[Proposition 7.3.4]{davisbook}.  We note here that, under this cellulation, the link of each vertex is $L$.  It follows that if $L$ is a triangulation of $\BS^{n-1}$, then $\gS$ is a topological $n$-manifold.  

%\begin{Proposition}\label{p:coxeter} There is a natural cell structure on $\gS$ so that 
%\begin{itemize}
%\item its vertex set is $W$, its 1-skeleton is the Cayley graph of $(W,S)$ and its 2-skeleton is a Cayley 2-complex.
%%\item each cell is a Coxeter cell.
%\item the link of each vertex is isomorphic to $L$ (the nerve of $(W,S)$) and so if $L$ is a triangulation of $\BS^{n-1}$, $\gS$ is a topological $n$-manifold.
%\item a subset of $W$ is the vertex set of a cell if and only if it is a spherical coset and
%\item the poset of cells is $W\cs$.
%\end{itemize}
%\end{Proposition}
%
\subsection{Previous results in $\ell^2$-homology}\label{ss:previous}
Let $L$ be a metric flag simplicial complex (see subsection \ref{ss:davis}), and let $A$ be a full subcomplex of $L$.  The following notation will be used throughout.
\begin{align}
\mfh_i(L) &:= \cH_i(\gS_L)\label{e:not1}\\
\mfh_i(A) &:= \cH_i(W_L\gS_A)\label{e:not2}\\
%\mfh_i(L, A)&:= \cH_i(\gS_L,W_L\gS_A)\label{e:not3}\\
\gb_{i}(A)&:= \dim_{W_L}(\mfh_i(A)).\label{e:not4}
%\gb_i(L,A)& := \dim_{W_{L}}(\mfh_i(L,A))
%\chi_\bq(L)&:= \sum (-1)^ib^i_\bq(A).
\end{align}
Here $\dim_{W_L}(\mfh_i(A))$ is the von Neumann dimension of the Hilbert $W_L$-module $W_L\gS_A$ and $\gb_{i}(A)$ is the $i^{\text{th}}$ $\ltwo$-Betti number of $W_L\gS_A$.  The notation in \ref{e:not2} and \ref{e:not4} will not lead to confusion since $\dim_{W_L}(W_L\gS_A)=\dim_{W_A}(\gS_A)$.  (See \cite{do2} and \cite{eckmann}).  

Given a simplicial complex $L$ and a full subcomplex $A\subset L$, we say that $A$ is \emph{$\ltwo$-acyclic}, if $\gb_{i}(A)=0$ for all $i$.

\noindent
\textbf{Bounded geometry.}  The following result is proved by Cheeger and Gromov in \cite{cheeggrom}.  Suppose that $X$ is a complete contractible Riemannian manifold with uniformly bounded geometry (i.e. its sectional curvature is bounded and its injectivity radius is bounded away from $0$.)  Let $\Gamma$ be a discrete group of isometries on $X$ with $\Vol(X/\Gamma)<\infty$.  Then $\dim_{\Gamma}(\cH_k(\underline{E\Gamma}))=\dim_{\Gamma}(\cH_{k}(X))$, where $\underline{E\Gamma}$ denotes a universal space for proper $\Gamma$ actions, and $\cH_{k}(X)$ denotes the space of $\Ltwo$-harmonic forms on $X$.  Of particular interest to us is the case where $X=\BH^{3}$.  For it is proved by Dodziuk in \cite{dodz} that the $\Ltwo$-homology of any odd-dimensional hyperbolic space, $\BH^{2k+1}$, vanishes.   

\noindent
\textbf{Euclidean Space.}  The Cheeger Gromov result also implies that if $\gS_{L}=\BR^n$ for some $n$, then $\mfh_{\ast}(L)$ vanishes.  

%In \cite{do2}, Davis and Okun prove Conjecture \ref{conj:singerc} for \emph{right-angled} Coxeter systems, $n\leq 4$.  We say that the Coxeter system $(W,S)$ is \emph{right-angled} if $m_{ss'}=2$ or $\infty$ for any $s\neq s'\in S$.  If $L$ denotes the nerve of $(W,S)$, then $(W,S)$ being right-angled means that each edge in $L$ is labelled $2$.  Some of their results extend to the general case.

\noindent\textbf{Joins.}  If $L=L_1\ast L_2$ where each edge connecting a vertex of $L_1$ with a vertex of $L_2$ is labeled $2$, then $W_L=W_{L_1}\times W_{L_2}$ and $\gS_L=\gS_{L_1}\times\gS_{L_2}$.  We may then use K\"unneth formula to calculate the (reduced) $\ltwo$-homology of $\gS_{L}$, and the following equation from \cite[Lemma 7.2.4]{do2} extends to our situation:
\begin{equation}\label{e:rt-angledjoin}
	\gb_{k}(L_{1}\ast L_{2})=\sum_{i+j=k}\gb_{i}(L_{1})\gb_{j}(L_{2}).
\end{equation}
\noindent\textbf{Suspensions.}  If $L=P\ast L_2$, where $P$ is two points not connected by an edge and each join edge is labeled with $2$, we call $L$ a \emph{right-angled suspension}.  $\gS_{P}=\BR$ and $\mfh_{i}(P)=0$ for all $i$ (\cite[Lemma 7.3.4]{do2}).  Then by equation \ref{e:rt-angledjoin}, $L$ is $\ltwo$-acyclic.  

%\begin{Lem}\label{l:rt-angledsusp} If $L$ is a right-angled suspension, then $L$ is $\ltwo$-acyclic.
%\end{Lem}
%
%\noindent\textbf{Cones.}  Suppose $L$ is the join of a point $p$ with a full subcomplex $K$ of $L$, i.e. $L$ is the cone on $K$.  If each edge joining $p$ to a vertex in $K$ is labelled $2$, then we call $L$ a \emph{right-angled cone}.  $\gS_{p}=\left[-1,1\right]$, and so $\gS_{L}=\left[-1,1\right]\times\gS_{K}$.  It is then clear from equation (\ref{e:rt-angledjoin}) that if $\gb_i(K)=0$, $\gb_i(L)=0$
%
\section{Andreev's theorem}\label{s:andreev}
In \cite{andreev2}, Andreev gives the necessary and sufficient conditions for abstract $3$-dimensional polytopes, with assigned dihedral angles in $\left(0,\frac{\pi}{2}\right]$, to be realized as (possibly ideal) convex polytopes in $\BH^3$ (these conditions are listed below, Theorem \ref{t:2:andreev}).  In order for this convex polytope to tile $\BH^3$, the assigned dihedral angles must be integer submultiples of $\pi$.  

Let $L$ be a labeled nerve of a Coxeter system, homeomorphic to $\BS^2$.  $K_L$ has assigned dihedral angles $\pi/m_{st}$ as discussed in Section \ref{ss:davis}.  So, if $K_L$ satisfies Theorem \ref{t:2:andreev}, then it follows that $\gS_{L}=\BH^3$.  However, it is possible that $K_L$ does not satisfy Andreev's theorem.  So, for the remainder of the paper, we will show how to apply Theorem \ref{t:2:andreev} to a modification $\left[L-T\right]$ of $L$.  (Here $\left[L-T\right]$ is a cell complex homeomorphic to $\BS^2$ with labeled edges.)  If $K_{\left[L-T\right]}$, with assigned dihedral angles corresponding to the edge labeling, satisfies Andreev's theorem, then it follows that $K_{\left[L-T\right]}$ is the strict fundamental domain for the action of a reflection group on $\BH^3$.

%Before stating Andreev's theorem, we state the Let $F_{1},\ldots,F_{s}$ be a sequence of faces of a three-dimensional polytope such that $F_{i}$ is adjacent only to $F_{i-1}$ and to $F_{i+1}$ (mod($s$)), and no three of them intersect at a vertex.  Then we say that the faces $F_{1},\ldots,F_{s}$ are an $s$-angled prismatic element.

\begin{Theorem}\label{t:2:andreev} \textup{(\cite[Theorem 2]{andreev2})} Let $P$ be an abstract three-dimensional polyhedron, not a simplex, such that three or four faces meet at every vertex.  The following conditions are necessary and sufficient for the existence in $\BH^{3}$ of a convex polytope of finite volume of the combinatorial type $P$ with the dihedral angles $\ga_{ij}\leq\frac{\pi}{2}$ (where $\ga_{ij}$ is the dihedral angle between the faces $F_{i},F_{j}$):
\begin{enumeratei}
	\item\label{i:i} If $F_{1}, F_{2}$ and $F_{3}$ are all the faces meeting at a vertex of $P$, then $\ga_{12}+\ga_{23}+\ga_{31}\geq \pi$; and if $F_{1}, F_{2}, F_{3}, F_{4}$ are all the faces meeting at a vertex of $P$ then $\ga_{12}+\ga_{23}+\ga_{34}+\ga_{41}=2\pi$.
	\item\label{i:ii} If three faces intersect pairwise but do not have a common vertex, then the angles at the three edges of intersection satisfy $\ga_{12}+\ga_{23}+\ga_{31}<\pi$.
	\item\label{i:iii} Four faces cannot intersect cyclically with all four angles $=\pi/2$ unless two of the opposite faces also intersect.  
	\item\label{i:iv} If $P$ is a triangular prism, then the angles along the base and top cannot all be $\frac{\pi}{2}$.  
	\item\label{i:v} If among the faces $F_{1},F_{2},F_{3}$ we have $F_{1}$ and $F_{2}$, $F_{2}$ and $F_{3}$ adjacent, but $F_{1}$ and $F_{3}$ not adjacent, but concurrent at one vertex and all three do not meet in one vertex, then $\ga_{12}+\ga_{23}<\pi$.
\end{enumeratei}
\end{Theorem}

\noindent
\textbf{The case where $L$ is the boundary of a $3$-simplex.} If $L$ is the boundary of a $3$-simplex, then $K_{L}$ is a $3$-simplex and we are unable to apply Andreev's theorem.  However, one can check that in this case $W_{L}$ is one of the groups listed in Figure $2.2$ or $6.2$ of \cite{humphreys} (n=4).  In fact, $\gS_L=\BE^3$ or $\gS_L=\BH^3$.  Therefore, if $L$ is the boundary of a $3$-simplex, then it is $\ltwo$-acyclic.   

%\begin{Lemma}\label{l:3-simplex} Let $T$ be a triangulation of $\BS^2$ isomorphic to the boundary of a $3$-simplex.  Then $T$ is $\ltwo$-acyclic.
%\end{Lemma}
%
\noindent
\textbf{Applying Andreev's theorem.} Suppose now that $L$ is not the boundary of a $3$-simplex.  If $s$ is a vertex of $L$, define the \emph{link of $s$ in $L$}, $L_{s}$, to be the subcomplex of $L$ consisting of all closed simplices which are contained in simplices containing $s$, but do not themselves contain $s$.   Define the \emph{star of $s$ in $L$}, $\St_{L}(s)$, to be the subcomplex of $L$ consisting of all closed simplices which contain $s$.  
%$L_{s}$ is a triangulation of $\BS^1$ and $\St_{L}(s)$ is a triangulation of a $2$-disk.  Note that if $L_s$ (respectively, $\St_{L}(s)$) is a full subcomplex of $L$, then it is the nerve of the subgroup $W_{L_{s}}$ ($W_{\St_{L}(s)}$) and $\gS_{L_{s}}$ ($\gS_{\St_{L}(s)}$) is a special subcomplex of $\gS_L$.  

The \emph{valence} of a vertex $s$ of $L$ is the number of vertices in its link.  We say that a vertex $s$ is \emph{3-Euclidean} if $s$ has valence $3$ and if $s_0, s_1, s_2$ are the vertices in this link, then 
\[\frac{\pi}{m_{s_0s_1}}+\frac{\pi}{m_{s_1s_2}}+\frac{\pi}{m_{s_2s_0}}=\pi.\]
We say that $s\in T$ is \emph{4-Euclidean}, if $s$ has valence $4$ and if $s_{0}, s_{1}, s_{2}, s_{3}$ are the vertices in this link, then $m_{s_{i}s_{i+1}}=2$ for $i=0,1,2,3$ (mod($4$)).  We'll say that the vertex $s$ is \emph{Euclidean} if it is either $3$- or $4$-Euclidean.

\begin{Lemma}\label{l:full} Let $s$ be a Euclidean vertex.
\begin{enumeratea}
	\item If $s$ is a $3$-Euclidean vertex, then $L_s$ and $\St_{L}(s)$ are full subcomplexes of $L$.  		\item  If $s$ is a $4$-Euclidean vertex and $L$ is not the suspension of a $3$-gon, then $L_s$ and $\St_{L}(s)$ are full subcomplexes of $L$.
\end{enumeratea} 
\end{Lemma}
\begin{proof} (a): This is immediate since $L$ is not the boundary of a $3$-simplex.  

(b): For a $4$-Euclidean vertex $s$, $L_s$ and $\St_L(s)$ can only fail to be full if $L$ is the suspension of a $3$-gon.
%Let $s$ be a $4$-Euclidean vertex, and let $s_0,s_1,s_2,$ and $s_3$ denote the vertices of the link of $s$.  Suppose that $s_{1}$ and $s_{3}$ are endpoints of a edge in $L$.  Then since $L$ is metric flag, $\{s_0,s_1,s_3\}$ and $\{s_2,s_1,s_3\}$ are $2$-simplices of $L$ and $L$ is the suspension of a $3$-gon, see Figure \ref{fig:susp3-gon2}.  This is a contradiction.  Thus $L_s$ and $\St_{L}(s)$ are full in $L$.
%
%\begin{figure}[h]
%	\hspace{2cm}
%	\input{susp3-gon2}
%	\caption{Suspension of a $3$-gon}
%	\label{fig:susp3-gon2}
%\end{figure}
\end{proof}

\begin{Lemma}\label{l:linkofeuclidean} Suppose that $L$ is not the suspension of a $3$-gon and let $s$ be a Euclidean vertex of $L$.  Then $L_{s}$ is $\ltwo$-acyclic.    
\end{Lemma}
\begin{proof} $\gS_{L_{s}}=\BR^{2}$.  Thus $\gb_{i}(L_{s})=0$ for all $i$.
\end{proof}

\begin{Lemma}\label{l:starofeuclidean} Suppose that $L$ is not the suspension of a $3$-gon and let $s$ be a Euclidean vertex of $L$.  Then $\St_{L}(s)$ is $\ltwo$-acyclic.
\end{Lemma}
\begin{proof} Suppose that $s$ is a $4$-Euclidean vertex.  Let $\left[\St\right]$ denote the complex obtained by capping off the boundary of $\St_{L}(s)$ with a square cell.  Then $K_{\left[\St\right]}$ clearly satisfies condition (\ref{i:i}) and satisfies conditions (\ref{i:ii})-(\ref{i:iv}) vacuously.  The only condition of Theorem \ref{t:2:andreev} that $K_{\left[\St\right]}$ may fail to meet is (\ref{i:v}).

If $K_{\left[\St\right]}$ does not satisfy this condition, then $\St_{L}(s)$ is a right-angled suspension and therefore $\ltwo$-acyclic.

%\begin{figure}[h]
%	\hspace{2cm}
%	\input{euclideanstar}
%	\caption{Star of Euclidean Vertex}
%	\label{fig:euclideanstar}
%\end{figure}
%
If $K_{\left[\St\right]}$ does satisfy condition (\ref{i:v}), then $K_{\left[\St\right]}$ can be realized as an ideal, convex polytope in $\BH^3$, the ideal vertex dual to the square face of $\left[\St\right]$.  The resulting reflection group is $W_{\St_{L}(s)}$, and by the results in Section \ref{ss:previous}, $\gb_{i}(\St_{L}(s))=0$ for all $i$.

Now suppose that $s$ is a $3$-Euclidean vertex.  If each edge in $(\St_{L}(s)-L_{s})$ is labeled $2$, then $\gS_{\St_{L}(s)}=\left[-1,1\right]\times \BR^2$ ($\gS_{s}=\left[-1,1\right]$), and by equation \ref{e:rt-angledjoin}, $\mfh_{i}(\St_{L}(s))$ vanishes.  Otherwise, let $\left[\St\right]$ denote the complex obtained by capping off the boundary of $\St_{L}(s)$ with a triangular cell.  The resulting reflection group, $W_{\St_{L}(s)}$, is one of the Coxeter groups shown in Figure 6.3 of \cite{humphreys}, the non-compact hyperbolic Coxeter groups ($n=4$).  It acts properly as a classical reflection group on $\BH^3$ with fundamental chamber $K_{\left[\St\right]}$, a simplex of finite volume with one ideal vertex corresponding to the added triangular face of $\left[\St\right]$.  Therefore $\gb_{i}(\St_{L}(s))=0$ for all $i$.  
\end{proof}

Let $C$ be a $3$-circuit in $L$ and let $s_0,s_1,s_2$ be the vertices in this circuit.  We say that $C$ is an \emph{empty Euclidean $3$ circuit} if $C$ is not the link of a vertex and if 
\[\frac{\pi}{m_{s_0s_1}}+\frac{\pi}{m_{s_1s_2}}+\frac{\pi}{m_{s_2s_0}}=\pi.\]
It follows from $L$ being metric flag that $C$ is a full subcomplex.

Let $C$ be a $4$-circuit in $L$.  Order the vertices in this circuit $s_{0},s_{1},s_{2},s_{3}$ so that $s_{i}$ and $s_{i+1}$ are connected by an edge of the circuit and $s_{i}$ and $s_{i+2}$ are not connected by an edge of the circuit ($i=0,1,2,3$ mod($4$)).  We say $C$ is an \emph{empty Euclidean $4$-circuit} if (a) $C$ is not the link of a vertex, (b) $C$ is not the boundary of the union of two adjacent $2$-simplices, and (c) $m_{s_{i}s_{i+1}}=2$ ($i=0,\ldots,3$ mod($4$)).  It follows from (b) and the fact that $L$ is metric flag that $C$ is a full subcomplex.

%\begin{Remark}\label{r:prismatic} Empty Euclidean $3$-circuits correspond to triangular prismatic elements in $K_L$ which do not satisfy condition (\ref{i:ii}) of Theorem \ref{t:2:andreev}.  Similarly, empty Euclidean $4$-circuits correspond to quadrangular prismatic elements in $K_{L}$ which do not satisfy condition (\ref{i:iii}).
%\end{Remark} 
%
\begin{Lemma}\label{l:do2:10.2.3} Suppose that $L$ has no empty Euclidean $4$-circuits and that $L$ is not the suspension of a $3$, $4$, or $5$-gon.  Then no two Euclidean vertices of $L$ are connected by an edge.
\end{Lemma}
\begin{proof} First, since $L$ is a metric flag, no two $3$-Euclidean vertices are connected by an edge.  

Second, suppose that $s$ and $s'$ are $4$-Euclidean vertices which are connected by an edge.  Then the star of that edge is the configuration pictured in Figure \ref{fig:v4vertices}.  The indicated vertices $v$ and $v'$ cannot coincide, since if they did $L$ would be the suspension of a $3$-gon.  The top and bottom vertices cannot be connected by an edge, since then $\{t,b,s'\}$ would be a spherical subset, and since $L$ is metric flag, it would not be a triangulation of $\BS^2$.  Let $C$ be the boundary of the star in the figure.  If $C$ is the boundary of two adjacent 2-simplices, then $L$ is the suspension of a $4$-gon.   If $C$ is the link of a missing vertex, then $L$ is the suspension of a $5$-gon.  Otherwise, $C$ is an empty Euclidean $4$-circuit, a contradiction.

%picture from do2 inserted here
\begin{figure}[h]
	\hspace{3.5cm}
	\input{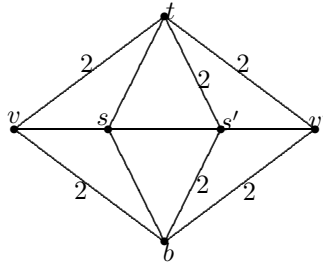}
	\caption{Two $4$-Euclidean vertices connected by an edge.}
	\label{fig:v4vertices}
\end{figure}

Lastly, suppose that $s$, a $3$-Euclidean vertex, and $s'$, a $4$-Euclidean vertex, are connected by an edge.  Then the star of that edge is the configuration pictured in Figure \ref{fig:v3v4vertices}.  Since $L$ is metric flag, $\{r,t,b\}$ is the vertex set of a simplex of $L$ and thus $L$ is the suspension of a $3$-gon, with $s$ and $r$ the suspension points, a contradiction.

\begin{figure}[h]
	\hspace{3.5cm}
	\input{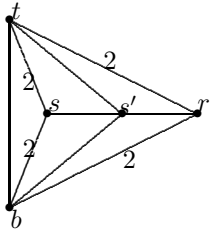}
	\caption{$3$-Euclidean vertex connected to $4$-Euclidean vertex.}
	\label{fig:v3v4vertices}
\end{figure}
 
\end{proof}

\begin{Lemma}\label{l:do2:10.2.4} Suppose $L$ is not the suspension of a $3$-gon.  Let $T$ be a set of Euclidean vertices of $L$, no two of which are connected by an edge.  Then $\gb_{i}(L)=\gb_{i}(L-T)$ for all $i$.
\end{Lemma}
\begin{proof} Let $s$ be a Euclidean vertex of $L$.  By Lemma \ref{l:full}, $L_{s}$ and $St_{L}(s)$ are full subcomplexes.  Consider the Mayer-Vietoris sequence:
\[\ldots\to \mfh_{i}(L_{s})\to \mfh_{i}(St_{L}(s))\oplus\mfh_{i}(L-s)\to \mfh_{i}(L)\to\mfh_{i-1}(L_{s})\to\ldots\]
By Lemmas \ref{l:linkofeuclidean} and \ref{l:starofeuclidean}, $\mfh_{i}(L_s)$ and $\mfh_{i}(St_{L}(s))$ vanish for all $i$.  The result follows.
\end{proof}

%Suppose that $L_1$ and $L_2$ are metric flag triangulations of $\BS^2$ and that $s_1$ and $s_2$ are Euclidean vertices with identical links (labels included).  Choose an identification of the link of $s_{1}$ with that of $s_{2}$.  Define a new triangulation $L_{1}\diamond L_{2}$ of $\BS^{2}$ by gluing together the 2-disks $L_{1}-s_{1}$ and $L_{2}-s_{2}$ along their boundaries.  
%
%For $j=1,2$, let $W_{j}$ denote the Coxeter group associated to $L_{j}$.  Then $L_{1}\diamond L_{2}$ is the nerve of the Coxeter group with generating set $(L_{1}-s_{1})\cup (L_{2}-s_{2})$ along with the existing relations on those sets plus relations identifying the generators corresponding to the vertices in the links of the $s_{j}$.  
%
Suppose $C$ is an empty Euclidean $3$- or $4$-circuit in $L$.  Then $C$ separates $L$ into two 2-disks, $D_{1}$ and $D_{2}$.  Let $L_{1}$ and $L_{2}$ denote the result of capping off $D_{1}$ and $D_{2}$, respectively (where ``capping off'' means adjoining a cone on the boundary, with edges each labeled $2$).  Let $s_1\in L_{1}$ and $s_2\in L_{2}$ denote the newly introduced cone points.  These are Euclidean vertices.  Since $C$ is an empty circuit, the two resulting triangulations, $L_{1}$ and $L_{2}$, each have fewer vertices than does $L$.  With this set up, we have the following lemma.

\begin{Lemma}\label{l:do2:10.2.7} $\gb_{i}(L)= \gb_{i}(L_{1})+\gb_{i}(L_{2})$ for all $i$.  As a result, $\mfh_{\ast}$ vanishes for $L$ if and only if it vanishes for both $L_{1}$ and $L_{2}$.
\end{Lemma}
\begin{proof} %$C$ is a full subcomplex of $L_{1}\diamond L_{2}$, and $L_{j}-s_{j}$ is a full subcomplex of $L_{1}\diamond L_{2}$ for $j=1,2$.  
Consider the Mayer-Vietoris sequence for $\gS_{L}$:
\[\ldots \mfh_{i}(C)\to \mfh_{i}(L_{1}-s_{1})\oplus \mfh_{i}(L_{2}-s_{2})\to \mfh_{i}(L)\to\mfh_{i-1}(C)\to\ldots\]
$\gS_{C}=\BR^2$ so $\mfh_{\ast}(C)$ vanishes.  Thus $\gb_{i}(L)=\gb_{i}(L_{1}-s_{1})+\gb_{i}(L_{2}-s_2)$ for all $i$.  By Lemma \ref{l:do2:10.2.4}, we have that $\gb_{i}(L_{j}-s_{j})=\gb_{i}(L_{j})$ for all $i$ and for $j=1,2$.  The desired equality is obtained.  
\end{proof}
\noindent
\textbf{Eliminating Euclidean vertices.}  Suppose $L$ is not the suspension of a $3$-, $4$-, or $5$-gon and that $L$ has no empty Euclidean $3$- or $4$-circuits.  Let $T$ denote the set of Euclidean vertices of $L$.  Consider a cellulation $\left[L-T\right]$ of $\BS^{2}$ obtained by replacing stars of $4$-Euclidean vertices by square cells and by replacing stars of $3$-Euclidean vertices by triangular cells.  Then either $L$ is a suspension of a $6$-gon formed from coning on the boundary of Figure \ref{fig:twoedges}, we refer to these as $L_6$-triangulations, or $\left[L-T\right]$ is a well-defined $2$-dimensional cell complex homeomorphic to $\BS^{2}$ with triangular and square faces in the strict sense that any nonempty intersection of two cells is a cell.  

%The fact that $\left[L-T\right]$ is homeomorphic to $\BS^2$ is clear.  To see that any nonempty intersection of two cells is a single cell, it suffices to check the possible intersections of the added triangular and square cells.  Two added squares cannot share all four of their edges (which are edges in $L$) since $L$ is not the suspension of a $4$-gon.  $L$ is not the suspension of a $3$-gon, so two added triangles cannot share all of their edges.  $\left[L-T\right]$ cannot be a 'pouch' (two vertices of $L$ would be connected by two different edges of $L$), so three edges of added squares cannot coincide, nor can two edges of added triangles.  By the same 'pouch' reasoning, two opposite edges of two added squares cannot coincide in the manner of a cylinder.  $L$ is homeomorphic to $\BS^2$, so the opposite edges of two added squares cannot coincide in the manner of a M\"obius band.  Two edges of an added triangle and two adjacent edges of an added square cannot coincide since no $3$-Euclidean vertex can have two edges of its link labelled $2$.  So, the only remaining possibility is that two adjacent edges of two added squares coincide.  Then in $L$, the star of these two edges is the configuration in Figure \ref{fig:twoedges}, where $s_1$ and $s_2$ indicate $4$-Euclidean vertices.
%
\begin{figure}[h]
	\hspace{3.5cm}
	\input{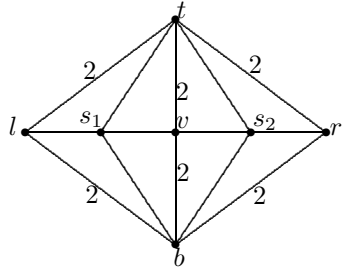}
	\caption{Stars of two $4$-Euclidean vertices intersecting in adjacent edges.}
	\label{fig:twoedges}
\end{figure}

%The top and bottom vertices, cannot coincide, since then two different edges connect the top vertex to $v$.  These vertices cannot be connected by an edge, for then $\{t,b,v\}$ would be the vertex set of a simplex of $L$, and $L$ triangulates $\BS^2$.  The vertices $r$ and $l$ cannot coincide, since $L$ is not the suspension of a $4$-gon, nor can they be connected by an edge, since $L$ is not the suspension of a $5$-gon.  Now let $C$ be the boundary of Figure \ref{fig:twoedges}.  $L$ contains no empty Euclidean $4$-circuits, so $C$ must be the link of a missing vertex $s_3$.  Then $L$ is the suspension of a $6$-gon shown in Figure \ref{fig:susp6-gon} with three $4$-Euclidean vertices $s_1, s_2$ and $s_3$.   
%
%\begin{figure}[h]
%	\hspace{2cm}
%	\input{susp6-gon}
%	\caption{Suspension of a $6$-gon with $4$-Euclidean vertices $s_1,s_2,s_3$.}
%	\label{fig:susp6-gon}
%\end{figure}
%
\begin{Lemma}\label{l:6-gon} $L_6$-triangulations are $\ltwo$-acyclic.
\end{Lemma}
\begin{proof} Any $L_6$-triangulation is the union of the star of a $4$-Euclidean vertex and the configuration in Figure \ref{fig:twoedges}, with intersection the boundary of the figure.  Figure \ref{fig:twoedges} can be decomposed as  $\St_{L}(s_1)\cup\St_{L}(s_2)$ with intersection being a right-angled suspension.  The desired result follows from Mayer-Vietoris.
\end{proof} 

%So, assuming that $L$ is not the suspension of a $6$-gon shown in Figure \ref{fig:susp6-gon}, $\left[L-T\right]$ is a well-defined cell complex.  Denote by $K_{\left[L-T\right]}$ the convex polytope with boundary complex dual to $\left[L-T\right]$.  We would like to apply Andreev's theorem to $K_{\left[L-T\right]}$.  However, it is possible that $K_{\left[L-T\right]}$ is a $3$-simplex, in which case $\left[L-T\right]$ is the boundary of a $3$-simplex.  Again, we handle this case without Andreev.  
%
%\noindent\textbf{$K_{\left[L-T\right]}$ a $3$-simplex.}  So by \ref{sss:finvol}, \ref{sss:hyperbolic} and \ref{l:do2:10.2.4}, 
%\begin{equation}\label{e:3-simplexv2}
%	\mfh_i(L)=0, \text{ for all } i\geq 0.
%\end{equation}
%
%So, assuming that $\left[L-T\right]$ is not the boundary of a $3$-simplex, we apply Andreev's theorem to $K_{\left[L-T\right]}$, yielding the following result.  
\begin{Theorem}\label{t:do2:10.3.1} Suppose that $L$ is not the boundary of a $3$-simplex and not an $L_6$-triangulation.  Suppose also that 
\begin{enumeratea}
	\item $L$ has no empty Euclidean $3$ or $4$-circuits, and
	\item $L$ is not the suspension of a $3$-, $4$- or $5$-gon.
\end{enumeratea}
Let $T$ denote the set of Euclidean vertices of $L$ and let $\left[L-T\right]$ be the cellulation of $\BS^2$ obtained by replacing stars of vertices in $T$ by triangular or square cells.  

Then $K_{\left[L-T\right]}$ can be realized as a (possibly ideal), convex polytope in $\BH^{3}$.  (The ideal vertices correspond to the square or added triangular faces of $\left[L-T\right]$, i.e. to the Euclidean vertices of $L$.)  The resulting classical reflection group is the Coxeter group $W_{L-T}$.  
\end{Theorem}

\begin{proof} If $K_{\left[L-T\right]}$ is a $3$-simplex, then the Coxeter group $W_{L-T}$ is one of the non-compact hyperbolic Coxeter groups shown in Figure $6.3$ of \cite{humphreys}.  Then $W_{L-T}$ acts on $\BH^3$ with fundamental chamber $K_{\left[L-T\right]}$, a finite volume simplex with ideal vertices dual to the added triangular faces of $\left[L-T\right]$.  Otherwise, we prove that $K_{\left[L-T\right]}$ satisfies the conditions of Andreev's theorem.

$\left[L-T\right]$ is a cell-complex with triangular and square faces, so $K_{\left[L-T\right]}$ has no more than three or four faces meeting at any vertex.  Condition (\ref{i:i}) is immediate under our hypothesis.  The remaining conditions refer to certain configurations of faces of the polytope.

$L$ contains no Euclidean vertices nor any empty Euclidean $3$- or $4$-circuits, so it follows that every $3$- or $4$-prismatic element in $K_{\left[L-T\right]}$ satisfies condition (\ref{i:ii}) or (\ref{i:iii}).  $L$ is not the suspension of a $3$-gon, so the only way $K_{\left[L-T\right]}$ can be a triangular prism is if $T$ is nonempty and $\left[L-T\right]$ is the suspension of a $3$-gon.  Then since we replaced the stars of some $3$-Euclidean vertices of $L$ with triangular cells whose three edge labels $m_1$, $m_2$ and $m_3$ have the property that $\frac{\pi}{m_1}+\frac{\pi}{m_2}+\frac{\pi}{m_3}=\pi$, we know that not every suspension line is labeled $2$.  Thus, $K_{\left[L-T\right]}$ satisfies condition (\ref{i:iv}).  

To verify condition (\ref{i:v}) we note that if two faces $F_{1}$ and $F_{3}$ of $K_{\left[L-T\right]}$ intersect at a vertex, but are not adjacent, then this vertex must have valence $4$.  So this vertex corresponds to a square cell of $\left[L-T\right]$, where each edge is labeled $2$, and the two faces are dual to opposite corners $f_{1}$ and $f_{3}$ of the square.  The configuration in condition (\ref{i:v}) has a third face, $F_{2}$, adjacent to both the previous two.  So its dual vertex, $f_{2}$, is connected to both $f_1$ and $f_3$ in $\left[L-T\right]$, and if either $m_{f_{1}f_{3}}\geq 3$ or $m_{f_{2}f_{3}}\geq 3$, condition (\ref{i:v}) is satisfied.  So suppose that both $m_{f_1f_2}$ and $m_{f_2f_3}$ equal $2$.  The square in $\left[L-T\right]$ corresponds to the star of Euclidean vertex in $L$.  If $v$ denotes one of the remaining corners of the square, then the vertices $f_1,f_2,f_3,v$ mark out a $4$ circuit in $L$, each of whose edges is labeled $2$.  Since $\left[L-T\right]$ is a well-defined cell-complex, this circuit cannot be the link of a missing vertex (two edges of added squares would coincide).  But $L$ does not contain empty Euclidean $4$-circuits, so $f_2$ is connected to $v$ ($f_1$ and $f_3$ are not connected because in the set-up of condition (\ref{i:v}), $F_1$ and $F_3$ are non-adjacent faces).  This means that $L$ contains a configuration pictured in Figure \ref{fig:v4vertices}, which according to Lemma \ref{l:do2:10.2.3} is impossible. \end{proof}

\begin{Lemma}\label{l:3,4,5-gons} Suppose that $L$ is the suspension of $3$-,$4$- or $5$-gon.  Then $\mfh_{\ast}(L)$ vanishes.
\end{Lemma}

\begin{proof} If $K_{L}$ satisfies the conditions of Andreev's theorem, then we are done.  So we consider cases in which $K_{L}$ does not satisfy the conditions of Andreev's theorem.

\textbf{Case 1:} Suppose that $L$ is the suspension of a $3$-gon.  Then the only conditions $K_{L}$ may fail to meet are (\ref{i:ii}) and (\ref{i:iv}).  Suppose $K_{L}$ does not satisfy (\ref{i:ii}).  Then the suspension points, $s$ and $s'$, are $3$-Euclidean vertices.  $L=St_{L}(s)\cup St_{L}(s')$, with $St_{L}(s)\cap St_{L}(s)=L_{s}$, the link of $s$.  Each piece is full in $L$ and $\ltwo$-acyclic, (Lemmas \ref{l:full}, \ref{l:linkofeuclidean} and \ref{l:starofeuclidean}).  So by Mayer-Vietoris, $L$ is $\ltwo$-acyclic.  
%\begin{figure}[placement]
%	\hspace{2cm}
%	\input{susp3-gon}
%	\caption{Suspension of a 3-gon}
%	\label{fig:susp3-gon}
%\end{figure}

Now suppose that $K_{L}$ satisfies (\ref{i:ii}) but does not satisfy (\ref{i:iv}).  Then in $L$, every suspension line is labeled $2$.  Thus $L$ is a right-angled suspension and $\mfh_{\ast}(L)$ vanishes.

\textbf{Case 2:} Suppose that $L$ is the suspension of a $4$-gon.  Then $K_{L}$ immediately satisfies conditions (\ref{i:i}), (\ref{i:ii}), (\ref{i:iv}) and (\ref{i:v}) of Andreev's theorem.  Suppose that $K_{L}$ does not satisfy condition (\ref{i:iii}).  Then $L$ has at least two $4$-Euclidean vertices, denote them $s$ and $s'$, and these can be arranged so that they are the suspension points.  Then $L=St_{L}(s)\cup St_{L}(s')$ with $St_{L}(s)\cap St_{L}(s')=L_{s}$.  Each piece is full in $L$ and $\ltwo$-acyclic.  The result follows from Mayer-Vietoris.

%\begin{figure}[placement]
%	\hspace{2cm}
%	\input{susp4-gon}
%	\caption{Suspension of a 4-gon}
%	\label{fig:susp4-gon}
%\end{figure}
%
\textbf{Case 3:} Lastly, suppose that $L$ is the suspension of a $5$-gon.  Again, $K_{L}$ immediately satisfies conditions (\ref{i:i}),(\ref{i:ii}),(\ref{i:iv}) and (\ref{i:v}) of Andreev's theorem.  If $K_{L}$ does not satisfy condition (\ref{i:iii}), then $L$ has at least one $4$-Euclidean vertex, $v$.  First, suppose that this vertex is not connected by an edge to any other $4$-Euclidean vertex.  Replace the star of this vertex with a square cell, and denote this cell complex  by $\left[L-v\right]$.  The only condition $K_{\left[L-v\right]}$ may fail to meet is (\ref{i:v}).  If $K_{\left[L-v\right]}$ satisfies this condition, then $v$ is the only Euclidean vertex, and $K_{\left[L-v\right]}$ can be realized as an ideal convex polytope in $\BH^3$.  (The ideal vertex corresponding to the square face of $\left[L-v\right]$.)  The resulting reflection group is $W_{L-v}$, and $\mfh_{i}(L-v)$ vanishes for all $i$.  Hence, by Lemma \ref{l:do2:10.2.4}, $\mfh_{i}(L)$ also vanishes.

If $K_{\left[L-v\right]}$ does not satisfy condition (\ref{i:v}), then $L$ decomposes as the $St_{L}(v)$ and the configuration in Figure \ref{fig:v4vertices}.  The intersection is $L_{v}$.  Figure \ref{fig:v4vertices} decomposes as the star of $s$, which is a Euclidean vertex, and the configuration in Figure \ref{fig:5-gonremains2v2}, a right-angled suspension.  The intersection is $L_s$.  All of these parts are full in $L$ and $\ltwo$-acyclic.  Use Mayer-Vietoris to determine that $L$ is $\ltwo$-acyclic.

%insert picture here
%\begin{figure}[placement]
%	\hspace{2cm}
%	\input{5-gonremains1}
%	\caption{Lemma \ref{l:3,4,5-gons}}
%	\label{fig:5-gonremains1}
%\end{figure}
%
\begin{figure}[placement]
	\hspace{3.5cm}
	\input{5-gonremains2v2}
	\caption{Lemma \ref{l:3,4,5-gons}}
	\label{fig:5-gonremains2v2}
\end{figure}

Next, suppose that $v$ is connected to another Euclidean vertex.  Then in $L$, there is at most one vertex $v'$ of the $5$-gon not connected to the suspension points by edges labeled $2$.  But, it is itself a $4$-Euclidean vertex.  So $L$ decomposes as the $\St_L(v')$ and a right-angled suspension, with intersection $L_v'$.  Each piece is full in $L$ and $\ltwo$-acyclic.  Use Mayer-Vietoris to determine that $L$ is $\ltwo$-acyclic. 
\end{proof}

\begin{Mainthm}\label{t:main} Let $L$ be a metric flag triangulation of $\BS^2$.  Then $\mfh_{i}(L)=0$ for all $i$.
\end{Mainthm}
\begin{proof} We may assume $L$ is not the boundary of a $3$-simplex, not an $L_6$-triangulation, and not the suspension of a $3$-,$4$- or $5$-gon.   If $L$ has no empty Euclidean $3$- or $4$-circuits, then by Theorem \ref{t:do2:10.3.1}, and the results in Section \ref{ss:previous}, $\mfh_{i}(L-T)$ vanishes for all $i$, where $T$ denotes the set of Euclidean vertices.  Hence, by Lemmas \ref{l:do2:10.2.3} and \ref{l:do2:10.2.4}, $\mfh_{i}(L)$ also vanishes.

In every other case, $L$ has an empty Euclidean $3$- or $4$-circuit which we can use to decompose $L$ as, $L=L_{1}\diamond L_{2}$.  Since $L_{1}$ and $L_{2}$ each have fewer vertices than does $L$, this process must eventually terminate.  The theorem follows from Lemma \ref{l:do2:10.2.7}.
\end{proof}

\bibliography{mybib}

\begin{thebibliography}{10}

\bibitem{andreev2}
E.~M. Andreev.
\newblock On convex polyhedra of finite volume in {L}oba\u{c}evski\u{i} space.
\newblock {\em Math. USSR Sbornik}, 12(2):255--259, 1970.

\bibitem{cheeggrom}
J.~Cheeger and M.~Gromov.
\newblock Bounds on the von {N}eumann dimension of $\ltwo$-cohomology and the
  {G}auss-{B}onnet theorem for open manifolds.
\newblock {\em Jounal of Differential Geometry}, 21:1--34, 1985.

\bibitem{davisannals}
M.~W. Davis.
\newblock Groups generated by reflections and aspherical manifolds not covered
  by {E}uclidean space.
\newblock {\em Annals of Mathematics}, 117:293--294, 1983.

\bibitem{davisbook}
M.~W. Davis.
\newblock {\em The Geometry and Topology of Coxeter Groups}.
\newblock Princeton University Press, Princeton, 2007.

\bibitem{davismoussong}
M.~W. Davis and G.~Moussong.
\newblock Notes on nonpositively curved polyhedra.
\newblock Ohio State Mathematical Research Institute Preprints, 1999.

\bibitem{do2}
M.~W. Davis and B.~Okun.
\newblock Vanishing theorems and conjectures for the $\ell^2$-homology of
  right-angled {C}oxeter groups.
\newblock {\em Geometry \& Topology}, 5:7--74, 2001.

\bibitem{dodz}
J.~Dodziuk.
\newblock ${L}^2$-harmonic forms on rotationally symmetric {R}iemannian
  manifolds.
\newblock {\em Proceedings of the American Mathematical Society}, 77:395--400,
  1979.

\bibitem{eckmann}
B.~Eckmann.
\newblock Introduction to $\ell^2$-methods in topology: reduced
  $\ell^2$-homology, harmonic chains, $\ell^2$-betti numbers.
\newblock {\em Israel Jounal of Mathematics}, 117:183--219, 2000.

\bibitem{humphreys}
J.~Humphreys.
\newblock {\em Reflection Groups and Coxeter Groups}.
\newblock Cambridge University Press, Cambridge, 1990.

\bibitem{schroedereven}
T.~A. Schroeder.
\newblock The $\ell^2$-homology of even {C}oxeter groups.
\newblock {\em {A}lgebraic \& {G}eometric {T}opology}, 9(2):1089--1104, 2009.
\newblock DOI number: 10.2140/agt.2009.9.1089.

\bibitem{schroedergeom}
T.~A. Schroeder.
\newblock {Geometrization of 3-dimensional Coxeter orbifolds and Singer's
  conjecture}.
\newblock {\em {G}eometriae {D}edicata}, 140(1):163ff, 2009.
\newblock DOI number: 10.1007/s10711-008-9314-5.

\end{thebibliography}

\end{document}